\documentclass[11pt,dvips,twoside,letterpaper]{article}
\usepackage{pslatex}
\usepackage{fancyhdr}
\usepackage{graphicx}
\usepackage{geometry}
 \RequirePackage[T1]{fontenc}

\def\figurename{Figure} 
\makeatletter
\renewcommand{\fnum@figure}[1]{\figurename~\thefigure.}
\makeatother

\def\tablename{Table} 
\makeatletter
\renewcommand{\fnum@table}[1]{\tablename~\thetable.}
\makeatother

\usepackage{amsmath}
\usepackage{amssymb}
\usepackage{amsfonts}
\usepackage{amsthm,amscd}

\newtheorem{theorem}{Theorem}[section]
\newtheorem{lemma}[theorem]{Lemma}

\newtheorem{proposition}[theorem]{Proposition}
\theoremstyle{definition}
\newtheorem{definition}[theorem]{Definition}
\newtheorem{example}[theorem]{Example}

\theoremstyle{remark}

\numberwithin{equation}{section}

\def\P{\mathbb P}
\def\R{\mathbb R}
\def\E{\mathbb E}

\def\E{\mathbb E}
\def\N{\mathbb N}

\begin{document}
\title{\bfseries\scshape{Reflected generalized backward doubly SDEs driven by L\'{e}vy processes and Applications}\thanks{This work is partially supported by TWAS Research Grants to individuals
No. 09-100 RG/MATHS/AF/\mbox{AC-I}--UNESCO FR: 3240230311.} }
\author{\bfseries\scshape Auguste Aman\thanks{augusteaman5@yahoo.fr}\\
{\it U.F.R Maths and informatique, Universit\'{e} de Cocody}, \\{\it 582 Abidjan 22, C\^{o}te d'Ivoire}}

\date{}
\maketitle

\begin{abstract}
In this paper, we study reflected generalized backward doubly
stochastic differential equations driven by Teugels martingales
associated with Lévy process (RGBDSDELs, in short) with one continuous barrier. Under uniformly Lipschitz
coefficients, we prove existence and uniqueness result by means
of the penalization method and the fixed point theorem. As an application, this study allows us to give a probabilistic representation for the solutions to a class of reflected stochastic partial differential integral equations (SPDIEs, in short) with a nonlinear Neumann boundary condition.
\end{abstract}


\vspace{.08in} \noindent \textbf{Keywords}: Reflected backward doubly SDEs, stochastic
PDIEs; L\'{e}vy process; Teugels martingale;
Neumann boundary condition.

\section{Introduction}
The theory of nonlinear backward stochastic differential equations (BSDEs, in short) have been first
introduced by Pardoux and Peng \cite{PP1}. They proved existence and uniqueness of the adapted processes $(Y,Z)$ solution of the following equation:
\begin{eqnarray}
Y_{t}&=&\xi+\int_{t}^{T}f(s,Y_{s},Z_{s})ds-\int_{t}^{T}Z_{s}dW_s,\,\ 0\leq t\leq
T,\label{a00}
\end{eqnarray}
when the terminal value $\xi$ is square integrable and the coefficient $f$ is Lipschitz in $(y, z)$ uniformly in $(t,\omega)$.

Mainly motivated by financial problems, stochastic control, stochastic games and probabilistic interpretation for
solutions to nonlinear partial differential equations (PDE, in short), the theory of BSDEs was developed
at high speed during the 1990. We refer the reader to survey article by El Karoui et al. \cite{Kal2}, Hamadène and Lepeltier \cite{HL}, Pardoux and Peng \cite{PP2}, Pardoux and Zhang \cite{PZ} and references therein.

In this dynamic, El Karoui et al. firstly introduced in \cite{Kal1} the notation of a solution of reflected backward
stochastic differential equations (RBSDEs, in short) with a continuous barrier. A solution for such equation associated with $(\xi,f,S)$, is a triple $(Y_t,Z_t,K_t)_{0\leq t\leq T}$ , which satisfies
\begin{eqnarray*}
Y_{t}&=&\xi+\int_{t}^{T}f(s,Y_{s},Z_{s})ds+K_T-K_t-\int_{t}^{T}Z_{s}dW_s,\,\ 0\leq t\leq
T,
\end{eqnarray*}
and $Y_t\geq S_t$ a.s. for any $t\in[0,T]$. The process $(K_t)_{0\leq t\leq T}$ is
non decreasing continuous whose role is to push upward the
process $Y$, in order to keep it above $S$. And it satisfies Skorokhod condition
\begin{eqnarray*}
\int_{0}^{T}(Y_s-S_s)dK_s=0.
\end{eqnarray*}
As shown in \cite{Kal1}, RBSDE's are a useful tool for the pricing of American options and the probabilistic representation for solutions to PDE's obstacle problem. Recall that many assumptions have been made to relax the assumption on the coefficient $f$ and the barrier; for instance, in \cite{M} Matoussi established the existence of a solution for RBSDE's with continuous and linear growth coefficient. Moreover, in \cite{H,HLM} RBSDEs with discontinuous barrier and  double barrier with continuous coefficients have been studied respectively. Also many authors studied RBSDEs replacing brownian motion by jumps process (see Hamadène and Ouknine \cite{HO} and references therein).

On the other hand, Pardoux and Peng study in \cite{PP3} the so-called backward doubly stochastic differential equations (BDSDEs, in short):
\begin{eqnarray}
Y_{t}&=&\xi+\int_{t}^{T}f(s,Y_{s},Z_{s})ds+\int_{t}^{T}g(s,Y_{s},Z_{s})dB_s-\int_{t}^{T}Z_{s}dW_s,\,\ 0\leq t\leq
T.\label{a00}
\end{eqnarray}
where $dW$ is a forward Itô integral and $dB$ the backward one. They prove among other a probabilistic representation for a class of quasi linear stochastic partial differential equations (SPDEs, in short).

In this paper, we study reflected generalized BDSDEs driven by Teugel martingale with respect to Lévy process (RGBDSDEL, in short) under Lipschiz coefficient, motivated by it application to obstacle problem for stochastic partial differential integral equations (SPDIEs, in short) and inspired by \cite{YO}.

The theory of BSDEs driven by Teugels martingales associated with Lévy process have been intensively study since Nualart and Schoutens prove in  \cite{NS1} the martingale representation theorem associated to L\'{e}vy process. They also derive in \cite{NS2} an existence and uniqueness result to BSDEs driven by Teugels martingales associated with Lévy process.  Since then, many others results have been derived. We refer the reader to \cite{Ot},\ \cite{QZ}, \cite{HY1} and reference therein. Note that all those studies were important from a pure mathematical point of view as well as in the world of finance. It could be used for the purpose of option pricing in a L\'{e}vy market and related PDIEs which provided an analogue of the famous Black and Scholes formula.

Roughly speaking, the present paper have two goal: first the existence and uniqueness of the solution to RGBDSDEL
\begin{eqnarray}
Y_{t}&=&\xi+\int_{t}^{T}f(s,Y_{s^-},Z_{s})ds+\int_{t
}^{T}\phi(s,Y_{s^-})dA_s+\int_{t}^{T}g(s,Y_{s^-},Z_s)\,dB_{s}\nonumber\\
&&-\sum_{i=1}^{m}\int_{t}^{T}Z^{(i)}_{s}dH^{(i)}_{s}+K_{T}-K_t,\,\ 0\leq t\leq
T,\label{a1}
\end{eqnarray}
is derived by means of the penalization method and the fixed point theorem when the terminal value $\xi$ is square integrable and the coefficients $f$ and $g$ is Lipschitz in $(y, z)$ uniformly in $(t,\omega)$. Furthermore, using this result we get a probabilistic representation for the solution of reflected SPDIE.

Due to the fact that the solution should be adapted to a family $(\mathcal{F}_t)$ which is not a filtration, the usual technics used in the classical reflected BSDEs (see e.g. \cite{Kal1}) does not work. Indeed, the section theorem cannot be easily used to derive that the solution stays above the obstacle for all time.

We give here a method which allows us to overcome this difficulty. The idea consists to start
from the basic RGBDSDEL with $g$ independent from $(y, z)$. We transform it to a RGBDSDEL with $g=0$,
for which we prove the existence and uniqueness of solutions by a penalization method. The section theorem is then
used in this simple context $(g=0)$ to prove that the solution of the RGBDSDEL with $g=0$, stays above the
obstacle at each time. The case where the coefficients $g$ depend on $(y, z)$ is then deduced by using a Banach fixed point theory.

The rest of paper is organized as follows. In Section 2, we state some notations, needed assumptions
and the definition of  solution to RGBDSDELs. Section 3, is devoted
to give our main results for existence and uniqueness for RGBDSDEL. Finally Section 4 point out a probabilistic representation of solutions to a class of reflected SPDIEs with a nonlinear Neumann boundary condition.

\section{Notations,  assumptions and definitions}
\setcounter{theorem}{0} \setcounter{equation}{0}
The scalar product of the space $\R^{d} (d\geq 2)$ will be denoted
by $<.>$ and the associated Euclidian norm  by $\|.\|$.

In what follows let us fix a positive real number
$T>0$. Let $(\Omega, \mathcal{F},\P,\mathcal{F}_{t},B_{t}, L_{t}: t\in [0, T])$ be a complete Wiener-L\'{e}vy space in $\R \times\R\backslash\{0\}$, with Levy measure $\nu$, i.e. $(\Omega, \mathcal{F},\P)$ is a complete probability space, $\{\mathcal{F}_{t}: t\in[0, T]\}$ is a right-continuous increasing family of complete sub $\sigma$-algebras of $\mathcal{F}$, $\{B_t: t\in [0, T]\}$ is a standard Wiener process in $\R$ with respect to $\{\mathcal{F}_{t}: t\in[0, T]\}$ and $\{L_{t}: t\in [0, T]\}$ is a $\R$-valued L\'{e}vy process independent of $\{B_{t} : t \in[0, T]\}$ and has only $m$ jumps size with non Brownian associated to a standard L\'{e}vy measure $\nu$ satisfying the following conditions:
$
\begin{array}{l}
\int_{\R}(1 \wedge y)\nu(dy) < \infty,
\end{array}
$

Let $\mathcal{N}$ denote the totality of $\P$-null sets of $\mathcal{F}$. For each $t\in[0, T]$, we define
\begin{eqnarray*}
\mathcal{F}_{t}= \mathcal{F}_{t}^{L}\vee \mathcal{F}_{t,T}^{B}\; \mbox{and}\; \widetilde{\mathcal{F}}_{t}= \mathcal{F}_{t}^{L}\vee \mathcal{F}_{T}^{B}
\end{eqnarray*}
where for any process $\{\eta_t\},\; \mathcal{F}_{s,t}^{\eta}=\sigma(\eta_{r}-\eta_s, s\leq r \leq t)\vee
\mathcal{N},\; \mathcal{F}_{t}^{\eta}=\mathcal{F}_{0,t}^{\eta}$.

We remark that ${\bf F}= \{\mathcal{F}_{t},\ t\in
[0,T]\}$ is neither increasing nor decreasing so that it does not a filtration. However $\widetilde{{\bf F}}= \{\widetilde{\mathcal{F}}_{t},\ t\in
[0,T]\}$ is a filtration.

We denote by $(H^{(i)})_{i\geq 1}$ the Teugels Martingale associated with the L\'{e}vy
process $\{L_t : t\in[0, T]\}$. More precisely
\begin{eqnarray*}
H^{(i)}=c_{i,i}Y^{(i)}+c_{i,i-1}Y^{(i-1)}+\cdot\cdot\cdot+c_{i,1}Y^{(1)}
\end{eqnarray*}
where $Y^{(i)}_{t}=L_{t}^{i}-\E(L^{i}_t)=L_{t}^{i}-t\E(L^{1}_t)$ for all $i\geq 1$ and $L^{i}_t$ are power-jump processes. That is $L^{1}_t=L_t$ and $L^{i}_t=\sum_{0<s<t}(\Delta L_s)^i$ for all $i\geq 2$, where $X_{t^-} = \lim_{s\nearrow t} X_s$ and $\Delta X_t = X_t - X_{t^-}$. It was shown in Nualart and Schoutens \cite{NS1} that the coefficients $c_{i,k}$ correspond to the orthonormalization of the polynomials $1, x, x^2,...$ with respect to the measure $\mu(dx)=x^{2}d\nu(x)+\sigma^{2}\delta_{0}(dx)$:
\begin{eqnarray*}
q_{i-1}(x)=c_{i,i}x^{i-1}+c_{i,i-1}x^{i-2}+\cdot\cdot\cdot+c_{i,1}.
\end{eqnarray*}
We set
\begin{eqnarray*}
p_i(x)=xq_{i-1}(x)=c_{i,i}x^{i}+c_{i,i-1}x^{i-1}+\cdot\cdot\cdot+c_{i,1}x^{1}.
\end{eqnarray*}
The martingale $(H^{(i)})_{i=1}^{m}$ can be chosen to be pairwise strongly orthonormal
martingale.

In the sequel, let\
$\{A_t,\ 0\leq t\leq T\}$\ be a continuous, increasing and
${\bf F}$-measurable real valued with bounded
variation on $[0,T]$ such that\ $A_0=0$.

For any $m\geq 1$, we consider the following spaces of processes:
\begin{enumerate}
\item $\mathcal{M}^{2}(\R^{m})$ denote the space of real valued, square integrable and $\mathcal{F}_{t}$-measurable processes $\varphi=\{\varphi_{t}:\;
t\in[0,T]\}$ such that
\begin{description}
\item $\|\varphi\|^{2}_{{\mathcal{M}}^{2}}=\E\int ^{T}_{0}\|\varphi_{t}\|^{2}dt<\infty$.
\end{description}
\item  $\mathcal{S}^{2}(\R)$ is the subspace of  $\mathcal{M}^{2}(\R)$ formed by the $\mathcal{F}_{t}$-measurable processes $\varphi=\{\varphi_{t}:\;
t\in[0,T]\}$ right continuous with left limit (rcll) such that
\begin{description}
\item $\displaystyle{\|\varphi\|^{2}_{\mathcal{S}^{2}}=\E\left(\sup_{0\leq t\leq T}|\varphi_{t}|
^{2}\right)<\infty}$.
\end{description}
\item  $\mathcal{A}^{2}(\R)$ is the set of $\mathcal{F}_{t}$-measurable, continuous, real-valued,
increasing process $\varphi=\{\varphi_{t}:\; t\in[0,T]\}$ such that $\varphi_0 = 0,\; \E|\varphi_T|^2 < \infty$
\end{enumerate}
Finally $\mathcal{E}^{2,m}=\mathcal{S}^{2}(\R)\times{\mathcal{M}}^{2}(\R^{m})\times\mathcal{A}^{2}(\R)$ endowed with the norm
\begin{eqnarray*}
\|(Y,Z,K)\|^{2}_{\mathcal{E}}=\E\left(\sup_{0\leq t\leq T}|Y_{t}|
^{2}+\int ^{T}_{0}\|Z_{t}\|^{2}dt+|K_{T}|^{2}\right).
\end{eqnarray*}
is a Banach space.

Next, we consider needed assumptions
\begin{itemize}
\item[$(\textbf{H1})$] $\xi$ is a square integrable random variable which is $\mathcal{F}_{T}$-measurable such  that for all $\mu>0$
$$ \mathbb{E}\left(e^{\mu A_T}|\xi|^2\right) < \infty. $$
\item[$(\textbf{H2})$]
$f:\Omega\times [0,T]\times\R\times \R^{m}\rightarrow \R$\
and $\phi:\Omega\times [0,T]\times\R \rightarrow \R,$ such that
\begin{itemize}
\item[$(a)$] There exist $\mathcal{F}_t$-measurable processes $\{f_t,\,
\phi_t,\,0\leq t\leq T\}$ with values in $[1,+\infty)$, and constants $\mu>0$ and $K>0$
such that for any $(t,y,z)\in [0,T]\times\R\times\R^{m}$ we have:
\begin{eqnarray*}
\left\{
\begin{array}{l}
f(t,y,z)\,\mbox{and}\;  \phi(t,y)\, \mbox{are}\, \mathcal{F}_t\mbox{-measurable processes},\\\\
|f(t,y,z)|\leq f_t+K(|y|+\|z\|),\\\\
|\phi(t,y)|\leq \phi_t+K|y|,\\\\
\displaystyle \E\left(\int^{T}_{0} e^{\mu
A_t}f_t^{2}dt+\int^{T}_{0}e^{\mu
A_t}\phi_t^{2}dA_t\right)<\infty.
\end{array}\right.
\end{eqnarray*}
\item[$(b)$] There exist constants $c>0, \beta<0$ and $0<\alpha<1$ such that for any $(y_1,z_1),\,(y_2,z_2)\in\R\times\R^{m}$,
\begin{eqnarray*}
\left\{
\begin{array}{l}
(i)\, |f(t,y_1,z_1)-f(t,y_2,z_2)|^{2}\leq c(|y_1-y_2|^{2}+\|z_1-z_2\|^{2}),\\\\
(ii)\; \langle y_1-y_2,\phi(t,y_1)-\phi(t,y_2)\rangle\leq
\beta|y_1-y_2|^{2},\\\\
(iv)\; |\phi(t,y_1)-\phi(t,y_2)|\leq
c|y_1-y_2|^{2},
\end{array}\right.
\end{eqnarray*}
\end{itemize}
\item[$(\textbf{H3})$]\
$g:\Omega\times [0,T]\times\R\times\R^m \rightarrow \R,$ such that
\begin{itemize}
\item[$(a)$] There exist $\mathcal{F}_t$-measurable process $\{g_t:\,0\leq t\leq T\}$ with values in $[1,+\infty)$, constants $\mu>0$ and $K>0$ such that for any $(t,y,z)\in
[0,T]\times\R\times\R^{m}$ we have:
\begin{eqnarray*}
\left\{
\begin{array}{l}
g(t,y,z)\,\,\mbox{is}\, \mathcal{F}_t\mbox{-measurable processes},\\\\
|g(t,y,z)|\leq g_t+K(|y|+|z|),\\\\
\displaystyle \E\left(\int^{T}_{0}e^{\mu A_t}g_t^{2}dt\right)<\infty.
\end{array}\right.
\end{eqnarray*}
\item[$(b)$]There exist constants $c>0$ and $0<\alpha<1$ such that for any $(y_1,z_1),\,(y_2,z_2)\in\R\times\R^{m}$,\,\\ $|g(t,y_1,z_1)-g(t,y_2,z_2)|^{2}\leq c|y_1-y_2|^{2}+\alpha\|z_1-z_2\|$.
\end{itemize}
\item[($\textbf{H4}$)] The obstacle $\left\{  S_{t},0\leq t\leq T\right\}$,
is a $\mathcal{F}_{t}$-measurable real-valued process satisfying
\begin{description}
\item $(i)\; \E\left(  \sup_{0\leq t\leq T}\left|S^+_{t}\right| ^{2}\right)  <\infty$,
\item $(ii)\; S_{T}\leq\xi$\;\; a.s.
\end{description}
\end{itemize}
\begin{definition}
We call solution of the RGBDSDEL \eqref{a1},
a $(\R\times\R^d\times\R+)$-valued process  $(Y,Z,K)$ which satisfied
$\eqref{a1}$
such that the following holds $\P$-a.s
\begin{description}
\item $(i)$ $(Y,Z,K)\in\mathcal{E}^{2,m}$
\item  $(ii)$ $Y_{t}\geq S_{t},\,\,\,\ 0\leq t\leq T$,\,\
\item  $(iii)$\ $\displaystyle \int_{0}^{ T}\left( Y_{t^-}-S_{t}\right)
dK_{t}=0$.
\end{description}
\end{definition}
\section{Main results}
\begin{lemma}(Comparison theorem see \cite{QZ})\label{Thm:comp}
Let $\xi^1$ and $\xi^2$ be two square integrable and $\widetilde{\mathcal{F}}_T$-measurable random variables, $f^1,f^2:[0,T]\times\Omega\times\R\times\R^{m}\rightarrow\R$ and
$\phi:[0,T]\times\Omega\times\R\rightarrow\R$ be three measurable functions. For $k=1,2$, let $(Y^k,Z^k)$ be a unique solution of the following BSDE:\\
$\left\{
\begin{array}{l}
Y_{t}^k=\xi^k+\int_{t}^{T}f^k(s,Y^k_{s^-},Z^k_{s})ds+\int_{t
}^{T}\phi(s,Y^k_{s^-})dA_s-\sum_{i=1}^{m}\int_{t}^{T}Z^{k(i)}_{s}dH^{(i)}_{s}\\\\
\E\left(\sup_{0\leq t\leq T}|Y^k_t|^2+\int_{0}^{T}\|Z_s^k\|^2ds\right)
\end{array}
\right.
$\\
We assume that
\begin{itemize}
\item $\xi^1\geq \xi^2, \ \P$-a.s.,
\item $f^1(t,Y^2,Z^2)\geq f^2(t,Y^2,Z^2),\; \P$-a.s.,
\item $\displaystyle\beta_t^i=\frac{f^1(t,Y_{t^{-}}^{2},\widetilde{Z}_{t}^{(i-1)})-
 f^1(t,Y_{t^{-}}^{2},\widetilde{Z}_{t}^{(i)})}{Z_t^{1(i)}- Z_{t}^{2(i)}}
  \mathbf{1}_{\left\{Z_t^{1(i)}\neq Z_{t}^{2(i)}\right\}},$
\end{itemize} where
$$\widetilde{Z}^{(i)}=\Big(Z^{2(1)},Z^{2(2)},...,Z^{2(i)},Z^{1(i+1)},...,Z^{1(m)}\Big)$$
satisfying $\displaystyle \sum_{i=1}^{m}\beta_t^i\Delta H_t^{(i)}>-1, \ dt\otimes d\P$-a.s.\\
Then, we have $Y_t^1\geq Y_t^2, \ \ a.s., \ \forall t\in[0,T]$.
 Moreover, if \ $ \xi^{1}> \xi^{2}$ or $f^1(t,Y^2,Z^2)> f^2(t,Y^2,Z^2)$ or $\phi^1(t,Y^2)> \phi^2(t,Y^2)$, a.s., we have $ Y_{t}^{1}> Y_{t}^{2},\ \ a.s.,\ \forall
t \in [0,T]$.
\end{lemma}
\begin{proof}
Let define
$$a_t=
[f^{1}(t,Y_{t^{-}}^{1},Z_{t}^{1})-f^{1}(t,Y_{t^{-}}^{2},Z_{t}^{1})]/
(Y_{t^{-}}^{1}-Y_{t^{-}}^{2})\textbf{1}_{\{Y_{t^{-}}^{1}\neq
Y_{t^{-}}^{2}\}}$$
$$b_t=
[\phi(t,Y_{t^{-}}^{1})-\phi(t,Y_{t^{-}}^{2})]/(Y_{t^{-}}^{1}-Y_{t^{-}}^{2})
\textbf{1}_{\{Y_{t^{-}}^{1}\neq Y_{t^{-}}^{2}\}};\hspace{1.55cm}$$ We note that $(a_t)_{t\in[0,T]}$ and
$(b_t)_{t\in[0,T]}$ are bounded measurable processes.

For $0\leq s \leq t \leq T$, let
$\Gamma_{s,t}=1+\displaystyle\int_s^t\Gamma_{s,r^{-}}dX_r,$ where
$$X_t=\int_{0}^{t}a_r dr+\int_{0}^{t}b_{r}dA_{r}+\sum_{i=1}^{m}\int_{0}^{t}\beta_{r}^{i}dH_{r}^{(i)}.$$
Then, we have (cf. Doléans-Dade exponential formula)
\begin{eqnarray}
 \Gamma_{s,t}=\exp\Big(\int_{s}^{t}dX_{r}-\frac{1}{2}\int_{s}^{t}\|
\beta_r\|^2dr\Big)\prod_{s<r\leq t}(1+\Delta X_r)\exp(-\Delta X_r)
\end{eqnarray}
with $\displaystyle\Delta X_t=\sum_{i=1}^{m}\beta_{t}^{i}\Delta H_{t}^{(i)}>-1$. Thus, for all
$0\leq s\leq t\leq T$, \ $\Gamma_{s,t}>0$.

Let denote $$\bar{\xi}=\xi^{1}-\xi^{2}, \ \ \ \ \bar{Y}_t
=Y_{t}^{1}-Y_{t}^{2}, \ \ \ \ \bar{Z}_t=Z_{t}^{1}-Z_{t}^{2}\hspace{1.6cm}$$
$$\hspace{0.6cm}\bar{f}_{t}=f^{1}(t,Y_{t^{-}}^{2},Z_{t}^{2})-f^{2}(t,Y_{t^{-}}^{2},Z_{t}^{2})
$$
and
$$\hspace{0.6cm}\bar{\phi}_{t}=\phi^{1}(t,Y_{t^{-}}^{2})-\phi^{2}(t,Y_{t^{-}}^{2}).
$$
Then
\begin{eqnarray*}\label{eq3l}
\bar{Y}_{t}&=&\bar{\xi}+\int_{t}^{T}[a_s \bar{Y}_{s^{-}}+ \sum_{i=1}^{m}\beta_s^i
\bar{Z}_{s}^{(i)}+\bar{f}_{s}]ds +\int_{t}^{T}[b_s\bar{Y}_{s^{-}}+\bar{\phi}_{s}]dA_{s}-\sum_{i=1}^{m}\int_{t}^{T}\bar{Z}_{s}^{(i)}dH_{s}^{(i)},\ \ \ \ \ t\in[0,T]
\end{eqnarray*}
Applying Itô's formula to $\Gamma_{s,r}Y_r$ from $r = t$ to $r = T$, it follows that
\begin{eqnarray*}
\Gamma_{s,t}\bar{Y}_t&=&\Gamma_{s,T}\bar{\xi}-\int_{t}^{T}\Gamma_{s,r^{-}}d\bar{Y}_r
-\int_{t}^{T}\bar{Y}_{r^{-}}d\Gamma_{s,r^{}}-\int_{t}^{T}d[\Gamma_{s,.},\bar{Y}]_{r}\\
&=&\Gamma_{s,T}\bar{\xi}+\int_{t}^{T}\Gamma_{s,r^{-}}
[\sum_{i=1}^{m}\beta_r^i
\bar{Z}_{r}^{(i)}+\bar{f}_{r}]dr+\int_{t}^{T}\Gamma_{s,r^{-}}\bar{\phi}_{r}dA_r-
\sum_{i=1}^{m}\int_{t}^{T}\Gamma_{s,r^{-}}\bar{Z}_{s}^{(i)}dH_{s}^{(i)}
\notag\\&& +
\sum_{i=1}^{m}\int_{t}^{T}\bar{Y}_{r^{-}}\Gamma_{s,r^{-}}\beta_{r}^{i}dH_{r}^{(i)}
-\sum_{i,j=1}^{m}\int_{t}^{T}\Gamma_{s,r^{-}}\beta_r^i\bar{Z}_{r}^{(j)}d[H^{i},H^{j}]_r\\
&=&\Gamma_{s,T}\bar{\xi}+\int_{t}^{T}\Gamma_{s,r^{-}}\bar{f}_{r}dr+\int_{t}^{T}\Gamma_{s,r^{-}}\bar{\phi}_{r}dA_r-
\sum_{i=1}^{m}\int_{t}^{T}\Gamma_{s,r^{-}}\bar{Z}_{s}^{(i)}dH_{s}^{(i)}
+\sum_{i=1}^{m}\int_{t}^{T}\bar{Y}_{r^{-}}\Gamma_{s,r^{-}}\beta_{r}^{i}dH_{r}^{(i)}.
\end{eqnarray*}
Taking conditional expectation w.r.t. $\widetilde{\mathcal{F}}_s$, is not hard to see that for $s=t$
\begin{eqnarray*}
\bar{Y}_{t}&=&\E\left(\Gamma_{t,T}\bar{\xi}+
\int_t^T\Gamma_{t,r^{-}}\bar{f}_{r}dr+\int_{t}^{T}\Gamma_{s,r^{-}}\bar{\phi}_{r}dA_r\ | \ \widetilde{\mathcal{F}}_{t}\right).
\end{eqnarray*}
Therefore $\bar{Y}_{t}\geq0$, i.e. $Y_t^1\geq Y_t^2$, a.s. Moreover if $\bar{\xi}> 0$ or $\bar{f}_t> 0$, a.s., then $\bar{Y}_{t}>0$, i.e. $Y_t^1> Y_t^2$, a.s.
\end{proof}

Now we state existence and uniqueness result.\newline
Firstly, we suppose $g$ independent from $(Y,Z)$ and consider RGBDSDEL:
\begin{eqnarray}
\left\{\begin{array}{ll}
Y_{t}=\xi+\int_{t}^{T}f(s,Y_{s^-},Z_{s})ds+\int_{t}^{T}\phi(s,Y_{s^-})dA_s+\int_{t}^{T}g(s)\,dB_{s}
-\sum_{i=1}^{m}\int_{t}^{T}Z^{(i)}_{s}dH^{(i)}_{s}+K_{T}-K_t,\\\\
Y_{t}\geq S_{t},\,\,\,\ 0\leq t\leq T,\\\\
(K_s)_{0\leq s\leq T}\; \mbox{is increasing, continuous and satisfies}\; \int_{0}^{ T}\left( Y_{s^-}-S_{s}\right)dK_s=0.
\end{array}
\right.
\label{a11}
\end{eqnarray}
\begin{proposition}
Under assumptions $({\bf H1})$,\ $({\bf H2})$ and
$({\bf H4})$, the basic RGBDSDEL \eqref{a11} has a unique solution.
\end{proposition}
\begin{proof}
In the sequel, $C$ denotes a strictly positive and finite constant which may take different values
from line to line.\newline
{\bf Existence.}\;
For each $n\in\N^{*}$, we set
\begin{eqnarray}\label{a0}
f_{n}(s,y,z)=f(s,y,z)+n(y-S_{s})^{-}.
\end{eqnarray}
By $\cite{HY1}$, let $(Y^{n},Z^{n})$ be a unique pair of process with values in $\R\times \R^{m}$ satisfying: $(Y^{n},Z^{n})\in S^2\times \mathcal{M}^{2}$
and
\begin{eqnarray}
Y_{t}^{n}&=&\xi+\int_{t}^{T}f_n(s,Y_{s^-}^{n},Z_{s}^{n})ds+\int_{t}^{T}\phi(s,Y_{s^-}^{n})dA_s\nonumber\\
&&+\int_{t}^{T}g(s)\,dB_{s}-\sum_{i=1}^{m}\int_{t}^{T}(Z_{s}^{n})^{(i)}dH^{(i)}_{s}. \label{h2}
\end{eqnarray}
Let \begin{eqnarray}
K^{n}_t=n\int_{0}^{t}(Y^n_{s^-}-S_s)^{-}ds\label{K}
\end{eqnarray}

{\it Step 1: A priori estimate}\\
We have
\begin{eqnarray*}
\sup_{n\in\N^{*}}\E\left( \sup_{0\leq t\leq T}\left|
Y_{t}^{n}\right| ^{2}+\int_{t}^{T}\left\| Z_{s}^{n}\right\| ^{2}
ds+|K^{n}_{T}|^{2}\right)<C.
\end{eqnarray*}
Indeed, by It\^{o}'s formula,  we have
\begin{align*}
&\E\left| Y_{t}^{n}\right|^{2}+\int_{t}^{T}\|Z_{s}^{n}\|^{2}ds\\
&\leq\left| \xi\right|
^{2}+2\E\int_{t}^{T}Y_{s^-}^{n}f(s,Y_{s^-}^{n},Z_{s}^{n}) ds
+2\E\int_{t}^{T}Y_{s^-}^{n}\phi(s,Y_{s^-}^{n})dA_s\nonumber\\
&+\E\int_{t}^{T}|g(s)|^{2}ds +2\E\int_{t}^{T} S_{s}dK_{s}^{n}.
\end{align*}
Using $(\textbf{H2})$ and the elementary inequality $2ab\leq \gamma
a^{2}+\frac{1}{\gamma} b^2,\ \forall\gamma>0$,
\begin{eqnarray*}
2Y_{s}^{n}f(s,Y_{s}^{n},Z_{s}^{n})&\leq&(c\gamma_1+\frac{1}{\gamma_1})|Y^{n}_{s}|^{2}
+2c\gamma_1\|Z^{n}_{s}\|^{2}+2\gamma_1 f_s^{2},\\
2Y_{s}^{n}\phi(s,Y_{s}^{n})&\leq& (\gamma_2-2|\beta|)|Y^{n}_{s}|^{2}+\frac{1}{\gamma_2}\phi_s^{2}.
\end{eqnarray*}
We choose $\displaystyle\gamma_1=\frac{1}{4c}$,
$\displaystyle\gamma_2=2|\beta|$ in the previous to obtain for all
$\varepsilon >0$
\begin{align}\label{a2}
&
\E\left|Y_{t}^{n}\right|^{2}+\frac{1}{2}\E\int_{t}^{T}\left\|Z_{s}^{n}\right\|^{2}ds \nonumber\\
&\leq
C\E\left\{|\xi|^{2}+\int_{t}^{T}|Y^{n}_{s}|^{2}ds+\int_{t}^{T}f_s^{2}ds+
\int_{t}^{T}\phi_s^{2}dA_s+\int_{t}^{T}|g(s)|^{2}ds\right\}\nonumber\\
& +\frac{1}{\varepsilon}\E\left(\sup_{0\leq s\leq
t}(S_{s}^+)^2\right)+\varepsilon\E \left(K_{T}^{n}-K_t^n\right)^2.
\end{align}
In virtue of (\ref{h2}) and $\eqref{K}$ we have
\begin{align}\label{kn1}
\E(K^n_T-K^{n}_t)^2\leq C\E\left\{|\xi|^{2}+\int_{t}^{T}f_s^{2}ds+
\int_{t}^{T}\phi_s^{2}dA_s+\int_{t}^{T}|g(s)|^{2}ds+\int_{t}^{T}
\left|Y_s^n\right|^2ds\right.\nonumber\\
\left.+\E\left(\sup_{0\leq s\leq
t}(S_{s}^+)^2\right)+\int_0^t
\left|Y_s^n\right|^2dA_s+\int_{t}^{T}\|Z_s^n\|^2ds\right\}
\end{align}
which, put in $\eqref{a2}$ together with Burkhölder-Davis-Gundy inequality provides
\begin{eqnarray*}
\E\left\{\sup_{0\leq t\leq T}
|Y_{t}^{n}|^{2}+\int_{t}^{T}\|Z_{s}^{n}\|^{2}ds+|K_{T}^{n}|^{2}\right\}
&\leq & C\E\left\{|\xi|^{2}+\int^{T}_{0}f_s^{2}ds+\int^{T}_{0}\phi_s^{2}dA_s\right.\nonumber\\
&&+\left.\int^{T}_{0}|g(s)|^{2}ds+\sup_{0\leq t\leq
T}(S_{t}^{+})^{2}\right\},
\end{eqnarray*}
provided that $\varepsilon $ is small enough.\\
{\it Step 2:} $Y_t\geq S_t$, a.s. $\forall\ t\in[0,T]$ where for all $0\leq t\leq T,\ Y_t=\sup_{n}Y^n_t$

Let us define
$$
\left\{
\begin{array}{ll}
&\displaystyle \overline{\xi}:=\xi+\int_{0}^{T}g\left(
s\right)  dB_{s}\\
& \displaystyle\overline{S}_{t}:=S_{t}+\int_{0}^{t}g\left(
s\right)  dB_{s}\\
& \displaystyle\overline{Y}_{t}^{n}:=Y_{t}^{n}+\int_{0} ^{t}g\left(
s\right)  dB_{s}.
\end{array}
\right.
$$
Hence,  according $\eqref{h2}$ we have
\begin{equation}\label{Ynbar}
\overline{Y}_{t}^{n}=\overline{\xi}+\int_{t}^{T}f\left(
s,Y_{s^-}^n,Z_s^n\right) ds+n\int_{t}^{T}\left(\overline{Y}_{s^-}^{n}
-\overline{S_{s} }\right) ^{-}ds+\int_{t}^{T}\phi\left(s,
Y_{s^-}^n\right) dA_s-\sum_{i=1}^{m}\int_{t}^{T}(Z_{s}^{n})^{(i)}dH^{(i)}_{s}.
\end{equation}
Let $(\widetilde{Y}^{n},\widetilde{Z}^{n})$ be a unique solution of the  GBDSDEL
\begin{eqnarray*}
\widetilde{Y}_{t}^{n} &=&\overline{S}_{T}+\int_{t}^{T}f\left(
s,Y_{s^-}^n,Z_s^n\right)
ds+n\int_{t}^{T}(\overline{S}_{s}-\widetilde{Y}_{s^-}^{n})ds
+\int_{t}^{T}\phi\left(s, Y_{s^-}^n\right)
dA_s-\sum_{i=1}^{m}\int_{t}^{T}(\widetilde{Z} _{s}^{n})^{(i)}dH^{(i)}_{s}.
\end{eqnarray*}
Since $\overline{S}_T\leq \overline{\xi}$, the previous comparison theorem shows that for every $n\geq 1$, $\overline{Y}^{n}_t\geq\widetilde{Y}_{t}^{n}$ a.s., for all $0\leq t\leq T$.\newline
Next, let $\sigma$ be a $\widetilde{\mathcal{F}}_t$-stopping time, and $\nu=\sigma\wedge T$. The sequence of processes $(\widetilde{Y}_{\nu }^{n})$ satisfies the equality
\begin{eqnarray*}
\widetilde{Y}_{\nu }^{n} &=&\E^{\widetilde{\mathcal{F}}_\nu}\left\{ e^{-n(T-\nu)}\overline{S}_{T}
+\int^{ T}_{\nu }e^{-n(\nu-s)}f(s,Y_{s^-}^{n},Z_{s}^{n})ds+n\int^{ T}_{\nu }e^{-n(\nu-s)}\overline{S}_{s}ds \right.\nonumber \\
&&\left.+\int^{ T}_{\nu }e^{-n(\nu-s)}\phi(s,Y_{s^-}^{n})dA_s\right\} \label{c'2}
\end{eqnarray*}
and therefore converges to $\overline{S}_\nu$ a.s. This implies that $S_\nu\leq Y_\nu$ a.s. It
follows from the section theorem (\cite{DM},\, p. 220) that for every $t\in[0,T],\; Y_{t}\geq S_{t}$ a.s.

{\it Step 3: Convergence of $(Y^n,Z^n)$}\\
\noindent Since $Y_{t}^{n}\nearrow Y_{t}$ a.s. for all $0\leq t\leq T$, using Fatou's lemma and step 1, we have
\begin{eqnarray*}
\E\left( \sup_{0\leq t\leq T }\left| Y_{t}\right| ^{2}\right)
<+\infty.
\end{eqnarray*}
Moreover, Lebesgue's dominated
convergence theorem  provide
\begin{eqnarray*}
\E\left( \int_{0}^{T }\left| Y_{s}^{n}-Y_{s}\right| ^{2} ds
\right)\longrightarrow 0,\,\,\ \mbox{as}\,\  n\rightarrow
\infty.\label{b9}
\end{eqnarray*}
Next, in virtue of step 2 we get, for\ $n \geq p$,
\begin{eqnarray*}
\E\left( \sup_{0\leq s\leq T}\left| Y_{s}^{n}-Y_{s}^{p}\right|
^{2}+\int_{0}^{T}\left\| Z_{s}^{n}-Z_{s}^{p}\right\| ^{2}ds+\sup_{0\leq s\leq T}\left| K_{s}^{n}-K_{s}^{p}\right|
^{2}\right)
\longrightarrow 0,\,\, \mbox{ as }\,  n,p\longrightarrow
\infty,\label{b12}
\end{eqnarray*}
which provides that the sequence of processes $(Y^{n},Z^{n},K^{n})$
is Cauchy in the Banach space $\mathcal{E}^{2,m}$. Consequently, there exists
a triplet $(Y,Z,K)\in\mathcal{E}^{2,m}$ such that
\begin{eqnarray*}
\E\left\{\sup_{0\leq s\leq T}\left| Y_{s}^{n}-Y_{s}^{{}}\right|
^{2}+\int_{0}^{T}\left\|Z^{n}_{s}-Z_{s}\right\|^{2}ds+\sup_{0\leq
s\leq T}\left| K_{s}^{n}-K_{s}^{{}}\right| ^{2}\right) \rightarrow
0,\mbox{ as }n\rightarrow \infty .
\end{eqnarray*}
{\it Step 4: The limit $(Y,Z,K)$ solve RGBDSDEL $\eqref{a11}$}\\
Since $(Y^{n},K^{n})$ converge to $(Y,K)$ in probability, the measure $dK^n$ converges to $dK$ weakly in probability, so that $\int_{0}^{T}(Y_{s^-}^{n}-S_{s})dK_{s}^{n}\rightarrow \int_{0}^{T}(Y_{s^-}-S_{s})dK_{s}$ in probability as $n\rightarrow \infty$. Obviously, $\int_{0}^{T}(Y_{s^-}-S_{s})dK_{s}\geq 0$,
while, on the other hand, for all $n\geq 0$, $\int_{0}^{T}(Y_{s^-}^{n}-S_{s})dK_{s}^{n}\leq 0$.
Hence $\int_{0}^{T}(Y_{s^-}-S_{s})dK_{s}= 0$, a.s. Finally, passing to the limit in $(\ref{h2}),\, (Y,Z,K)$ verifies (\ref{a1}).

{\bf Uniqueness.}\;\; Let $(\Delta Y,\Delta Z, \Delta K)$ be the difference between two arbitrary solutions. Since $\int_{t}^{T}(\Delta Y-\Delta S_s)d(\Delta K_s)=0$, the uniqueness follows using the same computation as above.
\end{proof}
\begin{theorem}\label{Th:exis-uniq}
Assume that $({\bf H1})$,\ $({\bf H2})$,\ $({\bf H3})$ and
$({\bf H4})$ hold. Then, RGBDSDEL \eqref{a1} has a unique solution.
\end{theorem}
\begin{proof} {\bf Existence.} In light of Proposition 3.2 and for $(\bar{Y},\bar{Z})\in\mathcal{S}^{2}(\R)\times \mathcal{M}^{2}(\R^m)$, let $(Y,Z,K)$ be a unique solution of RGBDSDEL:\\\\
$\left\{
\begin{array}{ll}
Y_{t}=\xi+\int^{T}_t f(s,{Y}_s,{Z}_s)ds
+\int_{t}^{T}\phi(s,Y_s)dA_s+\int_{t}^{T}g(s,\bar{Y}_s,\bar{Z}_s)\,dB_{s}
-\sum^{m}_{i=1}\int_{t}^{T}Z_{s}^{(i)}dH^{(i)}_{s}+K_T-K_t,\\
Y_t\geq S_t,\;\; \mbox{a.s.},\\
\int_0^T(Y_s-S_s)^{-}dK_s=0.
\end{array}\right.
$
We consider the mapping
$$
\begin{array}{lrlll}
\Psi:&\mathcal{S}^{2}(\R)\times \mathcal{M}^{2}(\R^m)&\longrightarrow&\mathcal{S}^{2}(\R)\times \mathcal{M}^{2}(\R^m)\\
&(\bar{Y},\bar{Z})&\longmapsto&(Y,Z)=\Psi(\bar{Y},\bar{Z}).
\end{array}
$$
Let $(Y,Z),\ (Y',Z'),\ (\bar{Y},\bar{Z})$ and $(\bar{Y'},\bar{Z'})$ in $\mathcal{S}^{2}(\R)\times \mathcal{M}^{2}(\R^m)$  such that
$(Y,Z)=\Psi(\bar{Y},\bar{Z})$ and $(Y',Z')=\Psi(\bar{Y'},\bar{Z'})$.
Putting $\Delta \eta=\eta-\eta'$ for any process $\eta$, we have
\begin{eqnarray*}
&&\E e^{-\mu t}|\Delta Y_t|^2+\E\int_{t}^T e^{-\mu s}\|\Delta Z_s\|^2ds\\
&&=2\E\int_{t}^T e^{-\mu s}\Delta
Y_s\left\{f(s,{Y}_{s^-},{Z}_{s})-f(s,{Y'}_{s^-},{Z'}_s)\right\}ds
+2\E\int_{t}^T e^{-\mu s} \Delta Y_{s}\left\{\phi(s,{Y}_{s^-})-\phi(s,{Y'}_{s^-})\right\}dA_s\\
 &&+2\E\int_{t}^T e^{-\mu s} \Delta Y_s d(\Delta K_s)+\int_{t}^Te^{-\mu
s}\left|g(s,\bar{Y}_{s^-},\bar{Z}_{s})-g(s,\bar{Y'}_{s^-},\bar{Z'}_{s})\right|^2ds-\mu\E\int_{t}^T
e^{-\mu s} \left|\Delta Y_s\right|^2 ds.
\end{eqnarray*}
Since $\displaystyle\E\int_{t}^T e^{-\mu s} \Delta Y_s d(\Delta K_s)\leq 0$ and using $({\bf H2})$-$({\bf H3})$, there exists constant $\alpha<\alpha'<1$ such that
\begin{eqnarray*}
&&(\mu-\gamma)\E\int_{t}^T
e^{-\mu s} \left|\Delta Y_s\right|^2 ds+\alpha'\E\int_{t}^T e^{-\mu s}\|\Delta Z_s\|^2ds\\
&&\leq c\E\int_{t}^T e^{-\mu s}\left|\Delta \bar {Y}_s\right|^2ds+\alpha\E\int_{t}^T e^{-\mu s}\left|\Delta \bar {Z}_s\right|^2ds,
\end{eqnarray*}
with $\gamma=\frac{c}{1-\alpha'}-1+\alpha$.

Choosing $\mu=\gamma+\alpha'c/\alpha$ and set $\bar{c}=\alpha' c/\alpha$, it follows from above that
\begin{eqnarray*}
&&\bar{c}\E\int_0^T
e^{-\mu s} \left|\Delta Y_s\right|^2 ds+\alpha'\E\int_0^T e^{-\mu s}\|\Delta Z_s\|^2ds\\
&&\leq \frac{\alpha}{\alpha'}\left(\bar{c}\E\int_0^T e^{-\mu
s}\left|\Delta \bar {Y}_s\right|^2ds+\alpha'\E\int_0^T e^{-\mu
s}\left|\Delta \bar {Z}_s\right|^2ds\right).
\end{eqnarray*}
Therefore $\Psi$ is a strict contraction on $\mathcal{S}^{2}(\R)\times \mathcal{M}^{2}(\R^m)$ equipped with the norm
\begin{eqnarray*}
\|Y,Z)\|^{2}=\bar{c}\E\int_0^T
e^{-\mu s} \left|Y_s\right|^2 ds+\alpha'\E\int_0^T e^{-\mu s}\|Z_s\|^2ds
\end{eqnarray*}
such that its unique fixed point is the solution of RGBDSDEL \eqref{a1}.

{\bf Uniqueness.}\; Assume $\left( Y_{t},Z_{t},K_{t}\right)_{0\leq t\leq T}$ and  $(Y_{t}^{\prime },Z_{t}^{\prime },K_{t}^{\prime })_{0\leq t\leq T}$ are two solutions of the RGBDSDEL $(\xi,f,g,\phi,S)$. We set $\Delta Y_{t}=Y_{t}-Y_{t}^{\prime},\, \Delta Z_{t}=Z_{t}-Z_{t}^{\prime }$ and
$\Delta K_{t}=K_{t}-K_{t}^{\prime }$.

Applying Itô's formula to $|\Delta Y|^2$ on the interval $[t,T]$ and taking
expectation on both sides, we have
\begin{eqnarray*}
&&\E\left| \Delta Y_{t}\right| ^{2}+\E\int_{t}^{T}\|\Delta Z_{s}\| ^{2}ds \\
&\leq&(4c^{2}+c+\frac{1}{2})\E\int_{t}^{T}|\Delta Y_{s}|^{2}ds+\frac{1}{2}\E\int_{t}^{T}\|\Delta Z_{s}\|^2ds.
\end{eqnarray*}
Hence by Gronwall's inequality, we derive $\E|\Delta Y_{t}|^{2}=0$ i.e $Y_{t}=Y^{\prime}_t$ a.s, for all $0\leq t\leq T$.
Therefore $Z_t=Z^{\prime}_t$ and $K_t=K^{\prime}_t$.
\end{proof}

\section {Connection to reflected stochastic PDIEs with nonlinear
Neumann boundary condition} \setcounter{theorem}{0}\setcounter{equation}{0}

In this section, we aim to show that the adapted solution of RGBDSDEL  is the solution of an obstacle problem for SPDIEs with a nonlinear Neumann boundary condition in the Markovian case under a regular assumptions on the coefficients.

We consider the L\'{e}vy process $L$ with no Brownian part and bounded
jump i.e $L_t=at+\int_{|z|\leq 1}z(N_t(.,dz)-t\nu(dz))$ where $N_t(\omega, dz)$ denotes the random measure such that $\int_{\Lambda}N_t(.,dz)$ is a Poisson process with parameter $\nu(\Lambda)$ for all set $\Lambda\;  (0\notin\Lambda)$. Without lost of generality, we suppose that
$\sup_t |\Delta L_t| \leq 1$. Then, for all $p\geq 1,\, \E|L_t|^p < \infty$. For more detail see \cite{Pr}, Theorem 34, page 25.

\subsection{A class of reflected diffusion process}
We now introduce a class of reflected diffusion process. For $\theta>0$, let $\Theta = (-\theta,\theta)$ and $e : [-\theta, \theta] \rightarrow \R$ such that $e(-\theta) = 1$ and $e(\theta) = -1$.\newline
Let $\sigma:\R\rightarrow\R$ be a uniformly bounded function which saisfies:
\begin{description}
\item $(i)\; |\sigma(x)-\sigma(x)|\leq K|x-x'|$ for every $x,x'\in\overline{\Theta}$,
\item $(ii)\; x+y\sigma(x){\bf 1}_{\{|y|\leq 1\}}\in\overline{\Theta}$ for every $x\in\overline{\Theta}$ and $y\in\R$,
\item $(iii)\; c(x)=c(pr(x))$ for all $x\in\R$, where $pr(.)$ denotes the orthogonal projection on $\overline{\Theta}$
\end{description}
As it shown in \cite{MR}, for every $(t,x)\in\overline{\Theta}$,  the process $(X^{t,x},\eta^{t,x})$ is a unique  solution of reflected SDE:
\begin{eqnarray}
\left\{
\begin{array}{l}
\P(X_s^{t,x}\in \overline{\Theta},\, s\geq t)=1\\\\
X_s^{t,x}=x+\int_{t}^s\sigma(X_{r^-}^{t,x})dL_r+\eta_s^{t,x},\; s\geq t,
\end{array}\right.
\label{RSDEJ1}
\end{eqnarray}
with $\eta^{t,x}_s=\int^{s}_t e(X_{r}^{t,x})d|\eta|_r$ with $|\eta^{t,x}|_s=\int^{s}_t{\bf 1}_{\{X_r^{t,x}\in\partial\Theta\}}d|\eta|_r$.

For our next purpose, let recall this needed Lemma. We refer the reader to \cite{NS2} for more detail.
\begin{lemma}
let $c:\Omega\times[0,T]\times\R\rightarrow\R$ be a measurable function such that
\begin{eqnarray*}
    |c(s,y)|\leq b_s(y^2\wedge|y|)\;\; a.s.,
\end{eqnarray*}
where $\{b_s, s\in [0, T]\}$ is a non-negative process such that $\E\int^T_0 b^2_sds <\infty$. Then, for each $0\leq t\leq T$, we have
\begin{eqnarray*}
\sum_{t\leq s\leq T}c(s,\Delta L_s)=\sum^{m}_{i=1}\int^{T}_{t}\langle c(s,.),p_i\rangle_{L^2(\nu)}dH^{(i)}_s+\int^{T}_{t}\int_{\R}c(s,y)d\nu(y)ds.
\end{eqnarray*}
\end{lemma}
\subsection{Feynman-Kac Formula}
Fix $T > 0$ and for all $(t,x)\in[0,T]\times\overline{\Theta}$, let $(X_s^{t,x},\eta_s^{t,x})_{s\geq t}$ denote the solution of the reflected SDE \eqref{RSDEJ1}. And we suppose now that the data $(\xi, f,\phi, g, S)$ of the RGBDSDEL take the form
\begin{eqnarray*}
\xi=l(X_T^{t,x}),\\
f(s, y, z)= f(s, X_s^{t,x}, y, z),\\
\phi(s,y)= \phi(s, X_s^{t,x}, y),\\
g(s, y, z)= g(t, X_s^{t,x}, y, z),\\
S_t = h(t, X_s^{t,x}).
\end{eqnarray*}
And we give the following assumptions:

Firstly, we assume that $h\in C^3([0,T]\times\overline{\Theta};\R)$ and $h(T,x)=l(x),\, \forall\,x\in\overline{\Theta},\; l\in C^3(\overline{\Theta};\R)$

Secondly, let $f\in C^3([0,T]\times\overline{\Theta}\times\R\times\R^m;\R),\; \phi\in C^3([0,T]\times\overline{\Theta}\times\R;\R)$ and $g\in C^3([0,T]\times\overline{\Theta}\times\R\times\R^m;\R)$ satisfy $({\bf H2})$ and $({\bf H3})$ respectively uniformly in $x\in\overline{\Theta}$.

It follows from the results of the Section 3 that, for all $(t,x)\in [0,T]\times\overline{\Theta}$ there exists a unique triple $(Y^{t,x}_s,Z^{t,x}_s,K^{t,x}_s)_{t\leq s\leq T}$ for the solution of the following RGBDSDEL:
\begin{eqnarray}
\left\{
\begin{array}{l}
(i)\;\E\left[\sup_{t\leq s\leq T}|Y^{t,x}_s|^2+\int_{t}^{T}\|Z^{t,x}_s\|^2ds\right];\\\\
(ii)\;
Y_{s}^{t,x}=l(X_T^{t,x})+\int_{s}^{T}f(r,X_{r}^{t,x},Y_{r}^{t,x},Z_{r}^{t,x})dr+\int_{s}^{T}
\phi(r,X_{r}^{t,x},Y_{r}^{t,x})d|\eta^{t,x}|_r+\int_{s}^{T}g(r,X_{r}^{t,x},Y_{r}^{t,x},Z_{r}^{t,x})\,dB_{r}\\\\
-\sum_{i=1}^{m}\int_{s}^{T}(Z^{t,x})^{(i)}_{r}dH^{(i)}_{r}+K^{t,x}_{T}-K^{t,x}_s,\,\ t\leq s\leq T;\\\\
(iii)\; Y^{t,x}_{s}\geq h(s,X^{t,x}_{s}),\,\,\,\ t\leq s\leq T;\\\\
(iv)\;(K^{t,x}_s)_{t\leq s\leq T}\; \mbox{is increasing, continuous and satisfies}\; \int_{t}^{ T}\left( Y^{t,x}_{s}-h(s,X_{s}^{t,x})\right)dK_s^{t,x}=0.
\end{array}\right.
\label{GBDSDEmarkovian}
\end{eqnarray}
We now consider the related obstacle problem for SPDIEs with a nonlinear Neumann boundary condition. Roughly speaking,
a classic solution of the obstacle problem is a random field such that $u(t,x)$ is
$\mathcal{F}_{t,T}$-measurable for each $(t,x)$ and $u\in C^{1,2}([0,T]\times\overline{\Theta};\R)$ which satisfies:
\begin{eqnarray}
\left\{
\begin{array}{l}
\displaystyle \min\left\{u(t,x)-h(t,x),\; \frac{\partial u}{\partial t}(t,x)+a'\sigma(x)
\frac{\partial u}{\partial x}(t,x)+f(t,x,u(t,x),(u^{(i)}(t,x))_{i=1}^{m})\right.\\\\
\left.\,\,\,\,\,\,\,\,\,\,\,\,\,\,\,\,\,\ +\int_{\R}u^1(t,x,y)d\nu(y)+
g(t,x,u(t,x),(u^{(i)}(t,x))_{i=1}^{m})\dot{B}_{t}\right\}=0,\,\,\
(t,x)\in[0,T]\times\Theta\\\\
\displaystyle e(x)\frac{\partial u}{\partial
x}(t,x)+\phi(t,x,u(t,x))=0,\,\,\ (t,x)\in[0,T]\times\{-\theta,\theta\},
\\\\ u(T,x)=l(x),\,\,\,\,\,\,\ x\in\Theta,
\end{array}\right.
\label{RSPDIE}
\end{eqnarray}
where
\begin{description}
\item $(i)\; a'=a+\int_{\{|y|\geq 1\}}y\nu(dy)$,
\item $(ii)\;  dB_t=\dot{B}_t$ denotes a white noise
\item $(iii)\; u^1(t,x,y)=u(t,x+y)-u(t,x)-\frac{\partial u}{\partial_x}(t,x)y$,
\item $(iv)\; u^{(1)}(t,x)=\int_{\R}u^1(t, x, y)p_1(y)\nu(dy) +\sigma(x)\frac{\partial u}{\partial x}(t, x)(\int_{\R} y^2\nu(dy))^{1/2}$,
\item $(v)\; u^{(i)}(t,x)=\int_{\R}u^1(t, x, y)p_i(y)\nu(dy),\; 2\leq i\leq m$.
\end{description}

We have
\begin{theorem}
Let $u$ the classic solution of SPDIE \eqref{RSPDIE}. Then the unique adapted solution of $(\ref{GBDSDEmarkovian})$ is given by
\begin{eqnarray*}
u(s,X^{t,x}_s)&=&Y^{t,x}_s\\
(Z^{t,x})^{(1)}_s &=&\int_{\R}u^1(t,X^{t,x}_{s^-}, y)p_1(y)\nu(dy) + \frac{\partial u}{\partial x} \sigma(X^{t,x}_{s^-})\left(\int_{\R}y^2\nu(dy)\right)^{1/2}\\
(Z^{t,x})^{(i)}_s &=&\int_{\R} u^{1}(t,X^{t,x}_{s^-},y)p_i(y)\nu(dy), \; 2\leq i\leq m.
\end{eqnarray*}
In particular $u(t,x)=Y^{t,x}_t$
\end{theorem}
\begin{proof}
For each $n\geq 1$, let
$\{^{n}Y_{s}^{t,x},{}^{n}Z_{s}^{t,x},\,\ t\leq s\leq T\}$ denote the
solution of the GBDSDEL
\begin{eqnarray*}
^{n}Y_{s}^{t,x}&=&l(X^{t,x}_{T})+\int^{T}_{s}f(r,X^{t,x}_{r^-},{}^{n}Y^{t,x}_{r^-},
{}^{n}Z^{t,x}_{r})dr+n\int^{T}_{s}({}^{n}Y^{t,x}_{r^-}-h(r,X^{t,x}_{r}))^{-}dr\nonumber\\
&&+\int^{T}_{s}\phi(r,X^{t,x}_{r^-},^{n}Y^{t,x}_{r^-})d|\eta^{t,x}|_r
\int^{T}_{s}g(r,X^{t,x}_{r^-},^{n}Y^{t,x}_{r^-},
{}^{n}Z^{t,x}_{r})dB_{r}-\sum^{m}_{i=1}\int^{T}_{s}
{}^{n}(Z^{t,x})^{(i)}_{r}dH^{(i)}_{r}. \label{c4}
\end{eqnarray*}
As it shown in \cite{HY1} we have
\begin{eqnarray*}
{}^{n}Y_s^{t,x} &=& u_n(s,X^{t,x}_s),\\
({}^{n}Z^{t,x})^{(1)}_t &=&\int_{\R}u_n^1(t,X_{t^-}, y)p_1(y)\nu(dy) + \frac{\partial u_n}{\partial x} \sigma(X^{t,x}_{s^-})\left(\int_{\R}y^2\nu(dy)\right)^{1/2}\\
({}^{n}Z^{t,x})^{(i)}_t &=&\int_{\R} u_n^{1}(t,X^{t,x}_{s^-},y)p_i(y)\nu(dy), \; 2\leq i\leq m,
\end{eqnarray*}
where $u_n$ is the classical solution of stochastic PDIE:
\begin{eqnarray*}
\left\{
\begin{array}{l}
\displaystyle \frac{\partial u_n}{\partial t}(t,x)+a'\sigma(x)
\frac{\partial u_n}{\partial x}(t,x)+f_n(t,x,u_n(t,x),(u_n^{(i)}(t,x))_{i=1}^{m})\\\\
\,\,\,\,\,\,\,\,\,\,\,\,\,\,\,\,\,\ +\int_{\R}u_n^1(t,x,y)d\nu(y)+
g(t,x,u_n(t,x),(u_n^{(i)}(t,x))_{i=1}^{m})\dot{B}_{t}=0,\,\,\
(t,x)\in[0,T]\times\Theta\\\\
\displaystyle e(x)\frac{\partial u_n}{\partial
x}(t,x)+\phi(t,x,u_n(t,x))=0,\,\,\ (t,x)\in[0,T]\times\{-\theta,\theta\},
\\\\ u_n(T,x)=l(x),\,\,\,\,\,\,\ x\in\Theta,
\end{array}\right.
\label{SPDIE}
\end{eqnarray*}
where $f_{n}(t,x,y,z)=f(t,x,y,z)+n(y-h(t,x))^{-}$.

Applying It\^{o}'s formula to $u_n(s,X_s)$, we obtain
\begin{eqnarray}
u_n(T,X^{t,x}_T)-u_n(s,X_s^{t,x})&=&\int_{s}^{T}\frac{\partial u_n}{\partial r}(r,X^{t,x}_{r^-})dr+\int_{s}^{T}e(X_r^{t,x})\frac{\partial u_n}{\partial x}(r,X^{t,x}_{r})d|\eta|_r\nonumber\\
&&+\int_{s}^{T}\sigma(X_{r^-})\frac{\partial u_n}{\partial x}(r,X^{t,x}_{r^-})dL_r\nonumber\\
&&+\sum_{s\leq r\leq T}[u_n(r,X^{t,x}_r)-u_n(r,X^{t,x}_{r^-})
-\frac{\partial u_n}{\partial x}(r,X^{t,x}_{r^-})\Delta X^{t,x}_r].\label{Ito}
\end{eqnarray}
Since $\Delta X^{t,x}_r=\sigma(X^{t,x}_{r^-})\Delta L_r$, applying Lemma 4.1 with $$c(r,y)=u_n(r,X^{t,x}_{r^-})+\sigma(X^{t,x}_{r^-})y)-u_n(r,X^{t,x}_{r^-})-\frac{\partial u_n}{\partial x}(r,X^{t,x}_{r^-})\sigma(X^{t,x}_{r^-})y,$$ we get
\begin{eqnarray}
\sum_{s\leq r\leq T}[u_n(r,X_r^{t,x})-u_n(r,X^{t,x}_{r^-})-\frac{\partial u_n}{\partial x}(r,X^{t,x}_{r^-})\Delta X^{t,x}_r]
&=&\sum_{i=1}^{m}\int_{s}^{T}\left(\int_{\R}u_n^{1}(r,X^{t,x}_{r^-},y)p_{i}(y)\nu(dy)\right)dH^{(i)}_r\nonumber\\
&&+\int_{s}^{T}\left(\int_{\R}u_n^{1}(r,X^{t,x}_{r^-},y)\nu(dy)\right)dr.\label{Applema}
\end{eqnarray}
Let us recall that
\begin{eqnarray}
L_t=Y^{(1)}_t+t\E L_1=\left(\int_{\R}y^2\nu(dy)\right)^{1/2}H^{(1)}+t\E L_1, \label{Levyproperty}
\end{eqnarray}
where $\E L_1=a+\int_{\{|y|\geq 1\}}y\nu(dy)$.
Hence, substituting $(\ref{RSDEJ1})$, $(\ref{Applema})$ and $(\ref{Levyproperty})$ into $(\ref{Ito})$ together with $(\ref{SPDIE})$ yields
\begin{eqnarray*}
&&l(X_T^{t,x})-u_n(t,X^{t,x}_s)\nonumber\\
&=&\int_{s}^{T}\left[\frac{\partial u_n}{\partial s}(r,X^{t,x}_{r^-})+(a+\int_{|y|\geq 1}y\nu(dy))\sigma(X_{r^-})\frac{\partial u_n}{\partial x}(r,X^{t,x}_{r^-})+\int_{\R}u_n^{1}(r,X^{t,x}_{r^-},y)\nu(dy)\right]dr\nonumber\\
&&+\int_{s}^{T}e(X_r)\frac{\partial u_n}{\partial x}(r,X^{t,x}_{r}){\bf 1}_{\{X^{t,x}_r\in\partial\Theta\}}d|\eta^{t,x}|_r\nonumber\\
&&+\int_{s}^{T}\left[\int_{\R}u_n^{1}(r,X^{t,x}_{r^-},y)p_{1}(y)\nu(dy)+\sigma(X^{t,x}_{r^-})\frac{\partial u_n}{\partial x}(r,X^{t,x}_{r^-})\left(\int_{\R}y^2\nu(dy)\right)^{1/2}\right]dH^{(1)}_r\nonumber\\
&&+\sum_{i=2}^{m}\int_{s}^{T}\left(\int_{\R}u_n^{1}(r,X^{t,x}_{r^-},y)p_{i}(y)\nu(dy)\right)dH^{(i)}_r.\nonumber\\
&=&-\int_{s}^{T}f(r,X^{t,x}_{r^-},u_n(r,X^{t,x}_r),(u_n(r,X^{t,x}_r))_{i=1}^{m})dr-n\int_{s}^{T}
(u_n(r,X^{t,x}_r)-h(r,X^{t,x}_r))^{-}dr\\
&&-\int_{s}^{T}g(r,X^{t,x}_{r^-},u_n(r,X^{t,x}_r),(u_n^{(i)}(t,x))_{i=1}^{m})dB_r-\int_{s}^{T}\phi(r,X^{t,x}_{r^-},u_n(r,X^{t,x}_r))d|\eta|_s\\
&&+\int_{s}^{T}\left[\int_{\R}u_n^{1}(r,X^{t,x}_{r^-},y)p_{1}(y)\nu(dy)+\sigma(X^{t,x}_{r^-})\frac{\partial u_n}{\partial x}(r,X^{t,x}_{r^-})\left(\int_{\R}y^2\nu(dy)\right)^{1/2}\right]dH^{(1)}_r\nonumber\\
&&+\sum_{i=2}^{m}\int_{t}^{T}\left(\int_{\R}u_n^{1}(r,X^{t,x}_{r^-},y)p_{i}(y)\nu(dy)\right)dH^{(i)}_s.
\end{eqnarray*}
Passing to the limit and using the previous section we get
\begin{eqnarray*}
&&u(t,X^{t,x}_s)-l(X_T^{t,x})\nonumber\\
&=&\int_{s}^{T}f(r,X^{t,x}_{r^-},u(r,X^{t,x}_r),(u(r,X^{t,x}_r))_{i=1}^{m})dr\\
&&+\int_{s}^{T}g(r,X^{t,x}_{r^-},u(r,X^{t,x}_r),(u(r,X^{t,x}_r))_{i=1}^{m})dB_r+K_T-K_t+\int_{s}^{T}\phi(r,X^{t,x}_{r^-},u(r,X^{t,x}_r))d|\eta|_s\\
&&-\int_{s}^{T}\left[\int_{\R}u^{1}(r,X^{t,x}_{r^-},y)p_{1}(y)\nu(dy)+\sigma(X^{t,x}_{r^-})\frac{\partial u}{\partial x}(r,X^{t,x}_{r^-})\left(\int_{\R}y^2\nu(dy)\right)^{1/2}\right]dH^{(1)}_r\nonumber\\
&&-\sum_{i=2}^{m}\int_{t}^{T}\left(\int_{\R}u^{1}(r,X^{t,x}_{r^-},y)p_{i}(y)\nu(dy)\right)dH^{(i)}_s,
\end{eqnarray*}
which get the desired result of the Theorem.
\end{proof}

In this follows, we give a example of reflected SPDIEs with a nonlinear Neumann
boundary condition.
\begin{example}
We consider the very special case where Lévy process $L$ is defined by $L_t = at+N_t-\alpha t$, where $N$ is Poisson processes with parameters $\alpha>0$. Then we have $H^{(1)}_t=\frac{\beta}{\sqrt{\alpha}}(N_t-\alpha t)$ and $H^{(i)}_t=0,\; i\geq 2$ (see \cite{NS2}). Moreover the reflected SPDIE \eqref{SPDIE} reduces to
\begin{eqnarray*}
\left\{
\begin{array}{l}
\displaystyle \min\left\{u(t,x)-h(t,x),\; \frac{\partial u}{\partial t}(t,x)+a'\sigma(x)
\frac{\partial u}{\partial x}(t,x)+f(t,x,u(t,x),\frac{\partial u}{\partial x}(t,x))\right.\\\\
\left.\,\,\,\,\,\,\,\,\,\,\,\,\,\,\,\,\,\ +\alpha u^1(t,x,\beta)+
g(t,x,u(t,x),\frac{\partial u}{\partial x}(t,x))\dot{B}_{t}\right\}=0,\,\,\
(t,x)\in[0,T]\times\Theta\\\\
\displaystyle e(x)\frac{\partial u}{\partial
x}(t,x)+\phi(t,x,u(t,x))=0,\,\,\ (t,x)\in[0,T]\times\{-\theta,\theta\},
\\\\ u(T,x)=l(x),\,\,\,\,\,\,\ x\in\Theta.
\end{array}\right.
\end{eqnarray*}
Then,
\begin{eqnarray*}
Y_s^{t,x} &=& u(s,X_s^{t,x}),\\
Z^{t,x}_s &=&\alpha u^1(t,X^{t,x}_{s^-},\beta)p_1(\beta) + \sqrt{\alpha}|\beta|\sigma(X^{t,x}_{s^-})\frac{\partial u}{\partial x}(s,X^{t,x}_{s^-}),
\end{eqnarray*}
where for each $(t,x)\in [0,T]\times\overline{\Theta},\; (Y^{t,x},Z^{t,x},K^{t,x})$ is the unique solution of the following
\begin{eqnarray*}
\left\{
\begin{array}{l}
(i)\;Y^{t,x}_{s}=l(X^{t,x}_T)+\int_{s}^{T}f(r,X^{t,x}_{r^-},Y^{t,x}_{r^-},Z^{t,x}_{r})dr+\int_{s
}^{T}\phi(r,X^{t,x}_{r^-},Y^{t,x}_{r^-})d|\eta^{t,x}|_r+\int_{s}^{T}g(r,X^{t,x}_{r^-},Y^{t,x}_{r^-},Z^{t,x}_{r})\,dB_{r}\\\\
-\int_{s}^{T}(Z^{t,x})_{s}d\widetilde{N}_r+K^{t,x}_{T}-K^{t,x}_s,\,\ t\leq s\leq
T\\\\
(ii)\; Y^{t,x}_{s}\geq h(s,X^{t,x}_{s}),\,\,\,\ t\leq s\leq T\\\\
(iii)\; \;(K^{t,x}_s)_{t\leq s\leq T}\; \mbox{is increasing, continuous and satisfies}\; \int_{t}^{ T}\left( Y^{t,x}_{s}-h(s,X_{s}^{t,x})\right)K_s^{t,x}=0.
\end{array}\right.
\label{GBDSDEmarkovian}
\end{eqnarray*}
with $\widetilde{N}_t=\frac{\beta}{\sqrt{\alpha}}(N_t-\alpha t)$.
\end{example}

\label{lastpage-01}

\begin{thebibliography}{99}

\bibitem{Bal} Barles, G.; Buckdahn, R.; Pardoux, E. Backward stochastic differential equations and integral-partial differential equations. {\it Stochastics Stochastics Rep.} {\bf 60} (1997), no. 1-2, 57-83.

\bibitem{DM} Dellacherie, C.; Meyer, P. Probabilities and Potential. North-Holland Mathematics Studies, {\bf 29}. North-Holland Publishing Co., Amsterdam-New York; North-Holland Publishing Co., Amsterdam-New York, 1978. viii+189 pp.  North Holland, (1978)

\bibitem{Ot} El Otmani, M. Generalized BSDE driven by a L\'{e}vy process. {\it J. Appl. Math. Stoch. Anal.} 2006, Art. ID 85407, 25 pp.

\bibitem{Kal1} El Karoui, N.; Kapoudjian, C.; Pardoux, E.; Peng, S.; Quenz, M. C. Reflected solution of backward SDE's, and related obstacle problem for PDE's, {\it Ann. Probab.} {\bf 25} (1997), no.2, 702-737.

\bibitem{Kal2}  El Karoui N., Peng S. and Quenez M. C., Backward stochastic differential equations in finance. {\it Math. Finance} {\bf 7} (1997), no. 1, 1-71.

\bibitem{H} Hamadène, S. Reflected BSDE's with discontinuous barrier and application. Stoch.
Stoch. Rep. 74 (2002), no. 3-4, 571¡596.

\bibitem{HL} Hamadène, S.; Lepeltier, J.-P. Zero-sum stochastic differential games and backward
equations. Systems Control Lett. 24 (1995), no. 4, 259¡263

\bibitem{HLM} Hamadène, S.; Lepeltier, J.-P.; Matoussi, A. Double barrier backward SDEs with
continuous coefficient. Backward stochastic differential equations (Paris, 1995-1996),
161-175, Pitman Res. Notes Math. Ser., 364, Longman, Harlow, 1997

\bibitem{HO} Hamadène, S.; Ouknine, Y. Reflected backward stochastic differential equation with jumps and random obstacle. Electron. J. Probab. 8 (2003), no. 2, 20 pp.

\bibitem{HY1} Hu L., Ren Y. Stochastic PDIEs with nonlinear Neumann boundary conditions and generalized backward doubly stochastic differential equations driven
by L\'{e}vy processes. {\it J. Comput. Appl. Math.} {\bf 229} (2009) 230-239.

\bibitem{K} Kunita, H. Stochastic flows and stochastic differential equations. Cambridge Studies in Advanced Mathematics, 24. {\it Cambridge University Press, Cambridge}, 1990.

\bibitem{M} Matoussi, Anis. Reflected solutions of backward stochastic differential equations with continuous coefficient. Statist. Probab. Lett. 34 (1997), no. 4, 347-354.

\bibitem{MR} Menaldi, J.; Robin, M. Reflected diffusion processes with jumps. {\it Ann. Probab.} {\bf 13} (1985), no. 2, 319-341.

\bibitem{NS1} Nualart, D.; Schoutens, W. Chaotic and predictable representations for L\'{e}vy processes. {\it Stochastic Process. Appl.} {\bf 90} (2000), no. 1, 109-122.

\bibitem {NS2} Nualart, D.; Schoutens, W. Backward stochastic differential equations and Feynman-Kac formula for L\'{e}vy processes, with applications in finance. {\it Bernoulli} {\bf 7} (2001), no. 5, 761-776.

\bibitem{PP1}  Pardoux E. and Peng S. Adapted solution of backward stochastic differential equation,
\text{ \sl Systems Control Lett.} {\bf 4} (1990), no.1, $55-61$.

\bibitem{PP2} Pardoux, E.; Peng, S. Backward stochastic differential equations and quasilinear parabolic partial differential equations. {\it Stochastic partial differential equations and their applications (Charlotte, NC, 1991)}, 200-217, Lecture Notes in Control and Inform. Sci., 176, {\it Springer, Berlin}, 1992.

\bibitem{PP3} Pardoux, E.; Peng, S. Backward doubly stochastic differential equations and systems of quasilinear SPDEs. {\sl Probab. Theory Related Fields} {\bf 98}, (1994), no.2, 209-227.

\bibitem{PZ}  Pardoux, E; Zhang, S. Generalized BSDEs and nonlinear Neumann boundary value problems, {\sl Probab. Theory Related Fields } {\bf
110} (1998), no.4, 535-558.

\bibitem{P} Peng, Shi Ge. Probabilistic interpretation for systems of quasilinear parabolic partial differential equations. {\it Stochastics Stochastics Rep.} {\bf 37} (1991), no. 1-2, 61-74.

 \bibitem{Pr} Protter, P. E.  Stochastic Integration and Differential Equations, 2nd ed., Stochastic Modeling and
Applied Probability, Springer, Berlin, 2005, Version 2.1.

\bibitem{QZ} Qing Zhou. On Comparison Theorem and Solutions of BSDEs for L\'{e}vy Processes. {\it Acta Mathematicae Applicatae Sinica, English Series} {\bf 23}, no. 3 (2007) 513522

\bibitem{RX} Ren, Y., Xia, N. Generalized reflected BSDE and obstacle problem for PDE with nonlinear Neumann boundary condition. {\sl Stoch. Anal. Appl.}
{\bf 24} (2006), no.5, 1013-1033.
\bibitem{YO} Y. Ren, M. El Otmani, Generalized reflected BSDEs driven by a Lévy process and an obstacle problem for PDIEs with a nonlinear Neumann boundary condition, {\it J. Comput. Appl. Math.} (2009), doi:10.1016/j.cam.2009.09.037

\end{thebibliography}
\end{document}